\documentclass[a4paper,12pt]{article}
\usepackage{amsmath, amsthm, amsfonts, amssymb, bbm}
\usepackage{graphicx, psfrag, color}
\usepackage{cite}
\usepackage{hyperref}

\newcommand{\Or}{\mathcal{O}}
\newcommand{\Ai}{\mathrm{Ai}}

\newcommand{\Pb}{\mathbb{P}}

\newcommand{\Id}{\mathbbm{1}}
\newcommand{\e}{\varepsilon}

\newcommand{\R}{\mathbb{R}}

\newcommand{\Z}{\mathbb{Z}}

\newcommand{\X}{\mathcal{X}}

\newtheorem{prop}{Proposition}[section]
\newtheorem{thm}[prop]{Theorem}
\newtheorem{lem}[prop]{Lemma}
\newtheorem{defin}[prop]{Definition}

\newtheorem{cla}[prop]{Claim}

\newtheorem{remark}[prop]{Remark}

\numberwithin{equation}{section}
\allowdisplaybreaks

\title{The Airy$_2$ process and the 3D Ising model}
\author{Patrik L. Ferrari\thanks{Institute for Applied Mathematics, University of Bonn, Germany; \texttt{ferrari@uni-bonn.de}} \and Senya Shlosman\thanks{Skolkovo Institute of Science and Technology, Moscow, Russia; Aix Marseille Univ, Universit\'e de Toulon, CNRS, CPT, Marseille, France; Inst. of the Information Transmission Problems, RAS, Moscow, Russia; \texttt{shlosman@gmail.com}}}

\date{}

\begin{document}
\sloppy

\maketitle

\begin{abstract}
The Ferrari-Spohn diffusion process arises as limit process for the 2D Ising model as well as random walks with area penalty. Motivated by the 3D Ising model, we consider $M$ such diffusions conditioned not to intersect. We show that the top process converges to the Airy$_2$ process as $M\to\infty$. We then explain the relation with the 3D Ising model and present some conjectures about it.
\end{abstract}

\section{Introduction and result}\label{SectIntro}

In this paper we consider $M$ non-intersecting Ferrari-Spohn diffusions and show that the top trajectory converges to the Airy$_2$ process in the $M\to\infty$ limit. The Ferrari-Spohn diffusion, denoted by $\tilde \X(t)$, is a diffusion process which first appeared in~\cite{FS04} as the limiting process of a Brownian motion conditioned to stay above a large circular barrier.
The infinitesimal generator of $\tilde \X(t)$ is given by
\begin{equation}\label{eqGenFS}
(L f)(x)=\frac12\frac{d^2 f(x)}{dx^2}+a(x)\frac{df(x)}{dx},
\end{equation}
where the drift is given by $a(x)=\frac{d}{dx}\ln(\Omega(x))$ with $\Omega(x)=\Ai(-\omega_1+x)$. Here, $-\omega_1$ is the right-most zero of the Airy function $\Ai$.

Another way to obtain this process is to consider a random walk conditioned to stay positive, which can be thought as having a hard-wall at the origin, and subjected to a penalty given in terms of the area under its trajectory. This model was studied in~\cite{ISV15} and it was motivated by the 2D Ising model, as we will discuss in more detail in Section~\ref{SectIsing2D}. One considers a discrete time random walk on $\Z$ with one-time transition probability $p(y)$ such that $\sum_{y\in\Z} y p(y)=0$ and $\sigma^2=\sum_{y\in\Z} y^2 p(y)<\infty$. Denote by $\mathbb{X}=(X_k)_{-N\leq k\leq N}$ the random walk starting from time $-N$ to time $N$ and by $\mathcal{P}_{N+}^{\mathsf{uv}}$ the set of trajectories such that $X_{-N}=\mathsf{u}$, $X_N=\mathsf{v}$, and $X_k>0$ for all $k$. Then, for $\lambda>0$, we consider the probability distribution on $\mathcal{P}_{N+}^{\mathsf{uv}}$ given by
\begin{equation}
\Pb_{N,\lambda}^{\mathsf{uv}}\left( \mathbb{X}\right) =\frac{1}{Z_{N,\lambda}^{\mathsf{uv}}}e^{-\lambda\sum_{j=-N}^{N}
X_{j}} \prod_{j=-N}^{N-1} p(X_{j+1}-X_j), \label{eqAreaTild}
\end{equation}
where $Z_{N,\lambda}^{\mathsf{uv}}$ is the normalization constant.

Without penalty, that is with $\lambda=0$, the process $\mathbb{X}$ under $\mathbb{P}_{N,0}^{\mathsf{uv}}$ fluctuates away from the wall by
$N^{1/2}$. On the other hand, if $\lambda>0$, then $\mathbb{X}$ remains localized as $N\to\infty$ (if $\mathsf{u,v}$ stay bounded). In the regime
$\lambda\to 0$ the typical distance of $\mathbb{X}$ from the wall is $\lambda^{-1/3}$, while the correlation distance along the interface is of order
$\lambda^{-2/3}$. Thus it makes sense to consider the process
$\alpha \lambda^{1/3}X_{[\beta t \lambda^{-2/3}]}$ for some constants $\alpha,\beta>0$.
For $\lambda=1/N$,
 $\alpha=\sigma^{2/3}2^{-1/3}$ and $\beta=\sigma^{-2/3} 2^{-2/3}$, is it proven in~\cite{ISV15} that\footnote{In~\cite{ISV15} the scaling was with $\alpha=\beta=1$, but we have chosen to add these to have convergence to $\tilde\X(t)$ without scaling factors in there.}, in the sense of finite-dimensional distributions,
\begin{equation}\label{eqRWtoFS}
\lim_{N\to\infty} \alpha N^{2/3}X_{[\beta t N^{2/3}]}=\tilde \X(t),
\end{equation}
provided that the initial and final points are $o(N^{1/3})$ from the wall. We refer to this scaling as (1/2/3) scaling.

Motivated by the 3D Ising model, see Section~\ref{slab} for a detailed discussion, the random walk model was extended in~\cite{IVW16} to several random walks as above but with the extra constraint to be non-intersecting. More precisely, one considers $M$ non-intersecting walks $\mathbb{X}^n=(X^n_k)_{-N\leq k\leq N}$, $n=1,\ldots,M$ subject to the same area tilt as in \eqref{eqAreaTild}\footnote{An extension to area penalty with prefactor $\lambda^i$ with $\lambda>1$ instead of constant $\lambda$ has been considered in~\cite{CIW19}.}. On top of it, one conditions on $X^n_k< X^{n+1}_k$, $n=1,\ldots,M-1$, $-N\leq k\leq N$. Under the scaling as in \eqref{eqRWtoFS}, it is shown that the collection $\{\mathbb{X}^n,1\leq n\leq M\}$ converges to the so-called Dyson Ferrari-Spohn diffusion process $\X(t)=(\X_1(t),\ldots,\X_M(t))$. It is a diffusion process in the Weyl chamber $W_{M,+}=\{0\leq x_1\leq \ldots \leq x_M\}\subset \R^M$ with zero-boundary conditions on $\partial W_{M,+}$. Its generator is given by
\begin{equation}\label{eqgenDFS}
(L_M f)(x)=\sum_{k=1}^M\left(\frac12 \frac{d^2}{d x_k^2}+a_{M,k}(x) \frac{d}{dx_k}\right) f(x),
\end{equation}
where $a_{M,k}(x)=\frac{d}{dx_k} \ln(\Omega_M(x))$ with the ground state $\Omega_M$ given by $\Omega_M(x_1,\ldots,x_M)=\det[\Ai(x_j-\omega_i)]_{1\leq i,j\leq M}$. Here, $-\omega_1>-\omega_2>\ldots$ denote the zeroes of the Airy function $\Ai$.

The result of this paper is to show that $\X_M(t)$ converges to the Airy$_2$ process as $M\to\infty$, which is a universal limit process in the Kardar-Parisi-Zhang universality class of stochastic growth models. The Airy$_2$ process was discovered in the study of the polynuclear growth model~\cite{PS02}.
\begin{thm}\label{thmMain}
Let $\tau_1<\tau_2<\ldots<\tau_m$ and $S_1,\ldots,S_m$ be fixed. Set $t_k=2\tau_k$ and $s_k=c_1 M^{2/3}+S_k$, with $c_1=\frac{3^{2/3}\pi^{2/3}}{2^{2/3}}$. Then,
\begin{equation}\label{eq51}
\lim_{N\to\infty} \Pb\bigg(\bigcap_{k=1}^m\{\X_M(t_k)\leq s_k\}\bigg) =\Pb\bigg(\bigcap_{k=1}^m\{{\cal A}_2(\tau_k)\leq S_k\}\bigg)
\end{equation}
with ${\cal A}_2$ is the Airy$_2$ process.
\end{thm}
The Airy$_2$ process is defined by its finite-dimensional distribution as follows.

\begin{defin}
The $m$-point joint distributions of the Airy$_2$ process ${\cal A}_2$ at times $\tau_1<\tau_2<\ldots<\tau_m$ are given by
\begin{equation}
\Pb\bigg(\bigcap_{k=1}^m\{{\cal A}_2(\tau_k)\leq S_k\}\bigg) = \det(\Id-\chi_s K_{\rm Ai})_{L^2(\R\times \{\tau_1,\ldots,\tau_m\})},
\end{equation}
where $\chi_s(x,\tau_k)=\Id_{x>S_k}$ and the extended Airy kernel $K_{\rm Ai}$ is given by
\begin{equation}
K_{\rm Ai}(\xi_i,\tau_i;\xi_j,\tau_j)= \left\{
\begin{array}{ll}
\int_0^\infty d\lambda e^{-\lambda(\tau_i-\tau_j)}\Ai(\xi_i+\lambda)\Ai(\xi_j+\lambda),&\textrm{ if }\tau_i\geq\tau_j,\\[0.5em]
-\int_{-\infty}^0 d\lambda e^{-\lambda(\tau_i-\tau_j)}\Ai(\xi_i+\lambda)\Ai(\xi_j+\lambda),&\textrm{ if }\tau_i<\tau_j.
\end{array}
\right.
\end{equation}
\end{defin}
In particular, the one-point distribution of the Airy$_2$ process is the so-called GUE Tracy-Widom distribution function, discovered in random matrices~\cite{TW94}.

In a zero-temperature case of the 3D-Ising corner, corresponding to a plane partition model (see~\cite{OR01}) which can be described via non-intersecting line ensembles, the appearance of the Airy$_2$ process at the edge was proven in~\cite{FS03}. However for the real 3D-Ising model the problem remains open.

For the special case of the one-point distribution, the original work~\cite{FS04} on Brownian motion conditioned to stay above a large circle, was extended in a physical level of rigor in~\cite{GS21} to several walks and a physical derivation of the one-point case is provided.

The rest of the paper is organized as follows. In Section~\ref{sectIsing} we discuss the relation between the model we studied and the Ising model. Furthermore, we present some conjectures on the Ising model related with our work. Finally, in Section~\ref{sectProof} we define in more detail the model and prove Theorem~\ref{thmMain}.

\paragraph*{Acknowledments:} The work of P.L. Ferrari was partly funded by the Deutsche Forschungsgemeinschaft (DFG, German Research Foundation) under Germany’s Excellence Strategy - GZ 2047/1, projekt-id 390685813 and Projektnummer 211504053 - SFB 1060. The work of S. Shlosman was partly funded by the ANR-RNF grant 20-41-09009.\\
This work has started during the scientific program on \emph{Randomness, Integrability, and Universality} of the Galileo Galilei Institute in 2022. We thank GGI for the hospitality and support. In particular we appreciate the activity and enthusiasm of Filippo Colomo, as an organizer of the meeting. We thank him for discussions on the subject of the present paper, as well as Alexey Borodin, Vadim Gorin, Kurt Johansson and Herbert Spohn.

\section{Relations with the Ising model}\label{sectIsing}

\subsection{Two-dimensional Ising model and Ferrari-Spohn diffusion}\label{SectIsing2D}
The following instance of the low-temperature (large $\beta$) Ising model was
considered in~\cite{SS96}. It is living in the square box $V_{N}=\left[
-N,N\right] \times\left[ -N,N\right] $ with $\left( -\right) $ boundary
conditions and under positive magnetic field $h=\frac{B}{N}$. It is proven
there that if $B$ is below certain critical value $B_{c}\left( \beta\right)
,$ then the $\left( -\right) $ boundary conditions win, and the box $V_{N}$
is filled with the $\left( -\right) $ phase. In the opposite case $B>B_{c}$
the magnetic field wins, and the box $V_{N}$ contains a big droplet of the
$\left( +\right) $-phase. This (random) droplet $\Gamma$ is pressed to the
sides of the box, and the $\left( -\right) $ phase fills only the four
corners of $V_{N},$ with total area $\sim c\left( \beta\right) \left\vert
V_{N}\right\vert ,$ where $c\left( \beta\right) \to0$ as
$\beta\to\infty$.

One would like to understand the fluctuations of $\Gamma$ in the various parts
of $V_{N}$. It turns out that in regions of $\Gamma$ where the distance to the corner is $O(N)$, the fluctuations
are of order $N^{1/2}$. But along the sides of $V_{N}$, where $\Gamma$ is
pressed to the walls, the fluctuations are only of order $N^{1/3},$ as was
explained first in~\cite{CLMST16}.

Comparing this random curve $\Gamma$ near the wall with the tilted random walk
of Section~\ref{SectIntro} one sees quite a similar picture, and expects that
again this part of $\Gamma,$ scaled by $N^{2/3}$ along the wall and by
$N^{1/3}$ in the orthogonal direction, converges, as $N\to\infty,$ to
the Ferrari-Spohn diffusion. Of course, the curve $\Gamma$ is not a graph of the
function; it has overhangs, plus its different parts interact, unlike the case
of the tilted random walks. Yet all these fine details of $\Gamma$ disappear
in the $(1/2/3)$ scaling, and the resulting limiting process is
again the same Ferrari-Spohn diffusion, see~\cite{IOSV22} for the details (compare with~\cite{GG21}). Its
generator is given by \eqref{eqGenFS}, with the $\Ai$ function scaled differently, namely replacing
$\Omega(x)$ by $\Ai\Big( \sqrt[3]{4Bm_\beta^\ast\chi_\beta^{1/2}}x-\omega_{1}\Big)$. Here appear various
quantities, characterizing the 2D Ising model for all subcritical temperatures; $m_\beta^\ast$ is the spontaneous magnetization, $\chi_\beta$ is the curvature of
the Wulff shape (at its bottom point), and $B$ is taken to be bigger than $B_{c}(\beta)$.

\subsection{Three dimensional model}\label{slab}
In this section a family of random lines will appear, which is the motivation for the multiple random walks models considered in~\cite{IVW16} and discussed in Section~\ref{SectIntro}.

The caricature of the crystal growth process was considered in
\cite{IS19}. The motivation was to study the dynamics of the large droplet of
the $\left( +\right) $-phase floating in the cubic box $V_{N}$ of size
$8N^{3}$ filled with the $\left( -\right) $-phase of the low-temperature 3D
Ising model, when the volume of the droplet (being of the order of $cN^{3}$) grows.

For that reason it was considered the 3D Ising model in $V_{N}$ with Dobrushin
boundary conditions, i.e. $\left( -\right) $-spins attached to the boundary
$\partial V_{N}$ in the upper half-space, and $\left( +\right) $-spins
attached to the boundary $\partial V_{N}$ in the lower half-space. These
boundary conditions force a (random) interface $\Gamma$ into $V_{N},$
separating the two phases. Further on, the model was considered in the
canonical ensemble, i.e. restricted to the spin configurations $\sigma$ with
the fixed value of the total magnetization $M$,
\begin{equation}
M\left( \sigma\right) =\sum_{x\in V_{N}}\sigma_{x}=C, \label{12}
\end{equation}
and the problem was to study the properties of $\Gamma$ as a function of the
value of $C$. A technical simplification was made in~\cite{IS19}, by passing to
the SOS-approximation of the model, so the level of the study corresponds to the random walk model in Section~\ref{SectIntro}, rather than to Section~\ref{SectIsing2D}.

The main result of~\cite{IS19} is that the interesting behavior of the surface
$\Gamma$ happens if one varies the constant $C$ on the scale $N^{2},$ i.e.
considers the dependence of $\Gamma$ as a function of the parameter $a,$ where
$C=aN^{2}$. Then the interface $\Gamma$ undergoes a sequence of transitions at
the values $0<a_{1}<a_{2}<..$. of the following nature:

\begin{enumerate}
\item For the values of $a\in\lbrack0,a_{1})$ the interface $\Gamma$ is
\textit{rigid}, i.e. it looks very similar to the horizontal plane $L=\left\{
\left( x,y,z\right) :z=0\right\} $ in $\mathbb{R}^{3}$. More precisely, the
density of locations where $\Gamma$ differs from $L$ (i.e. the relative area
of the symmetric difference $\Gamma\bigtriangleup L$) is low at low
temperatures: $\left\vert \Gamma\bigtriangleup L\right\vert \sim\alpha\left(
\beta\right) N^{2},$ with $\alpha\left( \beta\right) \to0$ as
$\beta\to\infty$. For any given location $\left( \bar{x},\bar
{y},0\right) \in L,$ the probability that the intersection of $\Gamma$ with
the vertical line $l_{\left( \bar{x},\bar{y}\right) }=\left\{ \left(
x,y,z\right) :x=\bar{x},y=\bar{y}\right\}$ is different from the point $\left( \bar{x},\bar
{y},0\right) \in L$ is of the order of
$e^{-4\beta}$. So, neglecting the local fluctuations, one
can say that the interface $\Gamma$ in the regime $a\in\lbrack0,a_{1})$ has
its height $h\left( \Gamma\right) $ equal to zero.

\item For the values of $a\in(a_{1},a_{2})$ the interface $\Gamma$ has one
\textit{monolayer} defined by the (random) contour $\gamma_{1}\subset L,$ which
means that the height of $\Gamma$ inside $\gamma_{1},$ i.e. $h\left(
\Gamma{\LARGE |}_{\mathrm{Int}\left( \gamma_{1}\right) }\right) ,$ equals
to $1,$ while $h\left( \Gamma{\LARGE |}_{\mathrm{Ext}\left( \gamma
_{1}\right) }\right) =0$ (again, neglecting the local fluctuations). This
monolayer is macroscopic in size, meaning that $\mathrm{diam}\left(
\gamma_{1}\right) \geq d_{1}\left( \beta\right) N,$ with $d_{1}\left(
\beta\right) >const>0$ uniformly in $a\in(a_{1},a_{2}).$ The segment
$(a_{1},a_{2})$ consists of two subsegments, $(a_{1},a_{2})=(a_{1},a_{3/2})\cup$ $(a_{3/2},a_{2})$, where the behavior of the interface
$\gamma_{1}$ differs slightly. For $a\in(a_{1},a_{3/2})$ the contour
$\gamma_{1}$ does not touch the boundary $\partial L$ and typically stays away
from it, at a distance $O\left( N\right) .$ For $a\in(a_{3/2},a_{2})$ the
contour $\gamma_{1}$ does touch the boundary $\partial L.$

\item For the values of $a\in(a_{2},a_{3})$ the interface $\Gamma$ has two
monolayers, defined by the pair of contours $\gamma_{1},\gamma_{2}\subset L,$
$\gamma_{2}\subset\mathrm{Int}\left( \gamma_{1}\right) .$ The height of
$\Gamma$ inside $\gamma_{2},$ i.e. $h\left( \Gamma{\LARGE |}_{\mathrm{Int}
\left( \gamma_{2}\right) }\right) ,$ equals to $2,$ $h\left(
\Gamma{\LARGE |}_{\mathrm{Int}\left( \gamma_{1}\right) \mathrm{\setminus
Int}\left( \gamma_{2}\right) }\right) =1,$ while again $h\left(
\Gamma{\LARGE |}_{\mathrm{Ext}\left( \gamma_{1}\right) }\right) =0$. Again,
$(a_{2},a_{3})=(a_{2},a_{5/2})\cup$ $(a_{5/2},a_{3}),$ and for $a\in
(a_{2},a_{5/2})$ the contour $\gamma_{2}$ is well inside $\gamma_{1},$ the
distance between them being typically $O\left( N\right) .$ In the remaining
regime $a\in(a_{5/2},a_{3})$ they are touching each other, the distance
between them being $O\left( N^{1/2}\right) .$

\item This process of creating extra monolayers continues, as the parameter
$a$ grows. But starting from some value $k=k\left( \beta\right) $ the
following extra feature takes place: for all $a\in(a_{k},a_{k+1})$ the
distances $\mathrm{dist}\left( \gamma_{i},\gamma_{j}\right) =o\left(
N\right) $ between all the $k$ nested contours $\gamma_{i}$. (Same holds even
for the Hausdorff distances, $\mathcal{H}\mathrm{dist}\left( \gamma
_{i},\gamma_{j}\right) $.) In other words, the subsegments $(a_{k},a_{k+1/2})\subset(a_{k},a_{k+1})$ become empty, once $k\geq k\left(
\beta\right) .$ As we just explained above, this is not the case for the few
initial values of $k,$ where the distance between the two top contours
$\mathrm{dist}\left( \gamma_{k-1},\gamma_{k}\right) =O\left( N\right) ,$
while $\mathrm{dist}\left( \gamma_{i},\gamma_{j}\right) =o\left( N\right)
$ for remaining $i,j<k,$ and the top monolayer stays away from the flock of
all the bottom ones.
\end{enumerate}

\subsection{Conjectures on the Ising model}\label{sectConjectures}
Here we will discuss the questions of the Wulff shapes and the way the random
interface $\Gamma$ is approximated by it. Let us first discuss the
analog of the setting of the Section~\ref{slab}, where the number of levels,
$k,$ of the interface, is not fixed, but grows with $N$. In other words, we
consider the Ising model as in the Section~\ref{slab}, but we put a different
canonical constraint:
\begin{equation}
M(\sigma) =\sum_{x\in V_{N}}\sigma_{x}=bN^{3},
\end{equation}
with $b>0$ some fixed constant. In order to avoid the situation of $\Gamma$
touching the top of the box $V_{N}$ one should take $b$ not too big. At low
temperatures, which is a regime we focus on, $b\leq1$ is good enough.

\paragraph{Conjectures}
\begin{enumerate}
\item \textit{Limit shape. }For every $b\in\left[ 0,1\right] ,$ every
$\beta$ large enough, there exists a (non-random) surface $W_{\beta,b}$ in the
cube $Q=\left[ -1,+1\right] ^{3},$ the boundary $\partial W_{\beta,b}$ of
which is the square, $\partial W_{\beta,b}=\partial Q\cap L$. This surface
$W_{\beta,b}$ is the typical shape of the random surface $\Gamma$. It means
that for every $\e>0$
\begin{equation}\label{17}
\lim_{N\to\infty}\Pb\left(\mathcal{H}\mathrm{dist}\left( \frac{1}{N}\Gamma,W_{\beta,b}\right)
>\e\right)=0,
\end{equation}
where $\mathcal{H}\mathrm{dist}$ is the Hausdorff distance.

\item \textit{The facets.} The surface $W_{\beta,b}$ is obtained by the
solution of the Wulff variation problem~\cite{DKS92}. It has certain height, $z_{\beta,b}$.
The intersection
\begin{equation}
W_{\beta,b}\cap \{L+z\}=\left\{
\begin{array}{ll}
\varnothing & \textrm{if }z>z_{\beta,b},\\
\textrm{1D curve} & \textrm{if }z<z_{\beta,b},
\end{array}
\right.
\end{equation}
while $W_{\beta,b}\cap \{L+z_{\beta,b}\}$
is a closed 2D region
with smooth boundary, which is the flat facet of $W_{\beta,b}$. The existence
of the facet is the corollary of the fact that the surface tension function
$\tau_{\beta}\left( \mathbf{n}\right)$, $\mathbf{n}\in\mathbb{S}^{2}$ has a
cusp at $\mathbf{n=}\left( 0,0,1\right) $ for all $\beta$ large enough. Let
$A_{\beta,b}>0$ be its area; clearly, $A_{\beta,b}\to4$ as
$\beta\to\infty$.

The random interface $\Gamma$ also has a (random) flat facet, in the following
sense. Define the height $h$ by
\begin{equation}
h(\Gamma)=\max\{z\,|\,\textrm{Area}\left( \Gamma\cap \{L+z\}\right) >\tfrac{1}{2}A_{\beta,b}N^{2}\},
\end{equation}
and if such value does not exist, put $h\left( \Gamma\right) =\infty$. We
define the facet $\Phi\left( \Gamma\right) \subset\mathbb{R}^{2}$ as the
intersection
\begin{equation}
\Phi\left( \Gamma\right) \cap\left\{ L+h\left( \Gamma\right) \right\} .
\end{equation}
We put $\Phi\left( \Gamma\right) =\varnothing$ for $h\left( \Gamma\right)=\infty$.

We conjecture that with probability going to $1$ as $\beta\to \infty$
the following happens:
\begin{itemize}
\item[(a)] $h\left(\Gamma\right)$ is finite,
\item[(b)] $\Phi\left(\Gamma\right)$ is indeed the facet of $\Gamma$, in the sense
that \textit{the next to }$\Phi\left( \Gamma\right) $ \textit{layer is very
small}:
\begin{equation}
\textrm{Area}\left( \Gamma\cap L+h\left( \Gamma\right) +1\right) <c\left(\beta\right) N^{2}
\end{equation}
with $c\left( \beta\right) \to 0$ as $\beta\to \infty$.
\end{itemize}

\item \textit{Airy$_2$ process.} Let $\partial\Phi$ be the exterior boundary of
the random facet $\Phi$. Our main conjecture is that the fluctuations of $\partial\Phi$ for the typical interface converge to the Airy$_2$
process, as $N\to\infty$.

The meaning of the statement is the following. Of course, for some
$\Gamma$ the curve $\partial\Phi$ can be weird or even empty. We claim that such curves
$\partial\Phi$ (and the interfaces $\Gamma$ themselves) are not typical: their probability goes to zero as
$N\to\infty$.
\end{enumerate}

The same kind of conjectures can be made for the case when one considers the
canonical low-temperature Ising model in the box $V_{N}$ with $\left(
-\right) $-boundary condition and with the canonical constraint
\begin{equation}
M\left( \sigma\right) =\sum_{x\in V_{N}}\sigma_{x}=\left( -m_{\beta}^{\ast
}+b\right) N^{3},
\end{equation}
where $b>0$. The canonical constraint produces an interface $\Gamma$ without
boundary, separating the $\left( +\right) $-phase inside $\Gamma$ from the
$\left( -\right) $-phase outside it, of linear size $\sim N$. In order that
$\Gamma$ can stay away from the walls $\partial V_{N}$ the parameter $b$
should not be too large; $b<1/2$ is fine.

The conjectures about the behavior of the typical $\Gamma$-s are very similar
to the above. They also have the asymptotic shape, given by the surface
$\tilde{W}_{\beta,b}$ in the cube $Q=\left[ -1,+1\right]^{3}$, this time
without boundary. It is given by the Wulff construction, see~\cite{DKS92} for
the 2D case and~\cite{Bod99,CP00} for the 3D case. The surface $\tilde{W}_{\beta,b}$ has six flat facets, in the above sense, while the typical interface
$\Gamma$ also has six random flat facets, of the linear size $\sim N$.
For the typical $\Gamma$ the fluctuations of the boundaries of the facets are
again given by the Airy$_2$ process.

We finish our conjecture list by pointing to the difference between the results of~\cite{DKS92} and of~\cite{Bod99,CP00}. In both cases the main claim is
that the random interface $\Gamma$ is close to its asymptotic shape $W$. But while in~\cite{DKS92} the closeness is measured in the Hausdorff distance, see
\eqref{17}, in~\cite{Bod99,CP00} it is measured in the weaker
$L^{1}$ sense, which replaces the distance $\mathcal{H}\mathrm{dist}(\frac{1}{N}\Gamma,W_{\beta,b})$ by the volume inside the closed
surface $\frac{1}{N}\Gamma\cup W_{\beta,b}$. (In the case without boundary one has to consider the volume of the symmetric difference $\textrm{Int}(\frac{1}{N}\Gamma) \bigtriangleup \textrm{Int}(\tilde{W}_{\beta,b})$
of the 3D bodies $\textrm{Int}(\frac{1}{N}\Gamma)$ and $\textrm{Int}(\tilde{W}_{\beta,b})$
properly shifted with respect to each other.) Our last conjecture is that for
low temperatures the convergence of the random interface $\Gamma$ to its
asymptotic shape $W$ holds in the stronger Hausdorff distance $\mathcal{H}\mathrm{dist}$. We also note that the part 2 of our conjectures does not hold
at zero temperature, see~\cite{BSS05}.

\section{Model and proof of the main result}\label{sectProof}

\subsection{Semigroup of the Ferrari-Spohn diffusion}
Let us introduce in more details the Ferrari-Spohn diffusion, see~\cite{FS04}. Consider the Airy operator
\begin{equation}
H_\Ai=-\frac{d^2}{dx^2}+x\textrm{ on }\R_+
\end{equation}
with Dirichlet boundary conditions at $0$. Let $-\omega_1>-\omega_2>\ldots$ the zeroes of the Airy function $\Ai$.
The normalized eigenfunctions of $H_\Ai$ are given by
\begin{equation}
\varphi_k(x)=\frac{\Ai(-\omega_k+x)}{(-1)^{k-1}\Ai'(-\omega_k)}, \quad k\geq 1,
\end{equation}
and have eigenvalues
\begin{equation}
H_\Ai \varphi_k(x)= \omega_k \varphi_k(x).
\end{equation}
The normalization comes from the identity $\int_{\R_+} dx (\Ai(x-a))^2=(\Ai'(-a))^2+a (\Ai(-a))^2$, where the last term vanishes at $a=\omega_k$. This identity is obtained by integration by parts and the identity $\Ai''(x)=x\Ai(x)$.
The ground state is $\Omega(x)=\varphi_1(x)$. So the Hamiltonian
\begin{equation}
H = -\frac12\frac{d^2}{dx^2}+\frac12x-\frac12\omega_1
\end{equation}
satisfies $H \Omega=0$. The Ferrari-Spohn diffusion $\tilde\X(t)$ is the stationary process on $\R_+$ obtained by the ground-state transformation (the Doob h-transform of $-H$): $(L f)(x)=-\Omega(x)^{-1} (H \Omega f)(x)$. The generator of $\tilde\X(t)$ is given by
\begin{equation}
(L f)(x)=\frac12\frac{d^2 f(x)}{dx^2}+a(x)\frac{df(x)}{dx}
\end{equation}
acting on smooth functions $f$, where
\begin{equation}
a(x)=\frac{d}{dx} \ln(\Omega(x))=\frac{\Ai'(-\omega_1+x)}{\Ai(-\omega_1+x)}.
\end{equation}
What we are interested in is the transition probability, namely the semigroup $T_t$ of $\tilde\X(t)$. It has integral kernel
\begin{equation}
T_t(x,y)=e^{\frac12\omega_1 t}\frac{\Omega(y)}{\Omega(x)} {\cal T}_t(x,y)
\end{equation}
with
\begin{equation}
{\cal T}_t(x,y)=(e^{-t \hat H})(x,y)=\sum_{k\geq 1}e^{-\frac12\omega_k t}\varphi_k(x)\varphi_k(y),
\end{equation}
and $\hat H=H+\frac12 \omega_1=\frac12 H_{\rm Ai}$.
The stationary distribution of $\tilde\X(t)$ has density on $\R_+$ given by
\begin{equation}
\rho(x)=\Omega(x)^2=\frac{\Ai(-\omega_1+x)^2}{\Ai'(-\omega_1)^2}.
\end{equation}

\subsection{Joint distributions of the Dyson Ferrari-Spohn diffusion}
Now we consider $M$ non-intersecting Ferrari-Spohn diffusions, $0<\X_1(t)<\X_2(t)<\ldots<\X_M(t)$. As shown in~\cite{ISV15,IVW16}, the ground state is given by the Slater determinant
\begin{equation}
\Omega_M(x_1,\ldots,x_M)=\det[\varphi_i(x_j)]_{1\leq i,j\leq M},
\end{equation}
and satisfies $H_M \Omega_M=0$ for
\begin{equation}
H_M =\sum_{k=1}^M \left(-\frac12\frac{d^2}{dx_k^2}+\frac12 x_k-\frac12\omega_k\right).
\end{equation}
It was shown in~\cite{ISV15,IVW16} that the generator of $\X(t)=(\X_1(t),\ldots,\X_M(t))$ is indeed \eqref{eqgenDFS} and its stationary distribution is given by
\begin{equation}\label{eqOnePt}
\Pb(\X_k(t)\in dx_k,1\leq k\leq M)=\frac{\Omega_M(x)^2}{Z_M} dx_1 \ldots dx_M,
\end{equation}
where $Z_M$ is the normalization constant. Moreover, the transition semigroup has kernel
\begin{equation}\label{eqTransition}
T_t(x,y)=e^{\sum_{k=1}^M \omega_k t}\frac{\Omega_M(y)}{\Omega_M(x)} \det\left[{\cal T}_t(x_i,y_j)\right]_{1\leq i,j\leq M}.
\end{equation}

From \eqref{eqOnePt} and \eqref{eqTransition} we get that the joint distributions at times $t_1<t_2<\ldots<t_m$ is given by
\begin{equation}\label{eqMultiPt}
\begin{aligned}
&\Pb\bigg(\bigcap_{n=1}^m\bigcap_{k=1}^M \{\X_k(t_n)\in dx_k^n\}\bigg)=\frac{\Omega_M(x)^2}{Z_M} \prod_{n=1}^{m-1}T_t(x_i^n,x_j^{n+1})_{1\leq i,j\leq n} \prod_{n=1}^m\prod_{k=1}^M dx_k^n\\
& =\frac{1}{\tilde Z_M} \det[e^{-\frac12 \omega_i t_1}\varphi_i(x_j^1)]_{1\leq i,j\leq M} \prod_{n=1}^{m-1} \det\big[{\cal T}_{t_{n+1}-t_n}(x_i^n,x_j^{n+1})\big]_{1\leq i,j\leq M} \\&\quad\times \det[e^{\frac12\omega_i t_M}\varphi_i(x_j^M)]_{1\leq i,j\leq M} \prod_{n=1}^m\prod_{k=1}^M dx_k^n\\
& =\frac{1}{\tilde Z_M} \det[\Phi^1_i(x_j^1)]_{1\leq i,j\leq M} \prod_{n=1}^{m-1} \det\big[{\cal T}_{t_{n+1}-t_n}(x_i^n,x_j^{n+1})\big]_{1\leq i,j\leq M} \\&\quad\times \det[\Psi^M_i(x_j^M)]_{1\leq i,j\leq M} \prod_{n=1}^m\prod_{k=1}^M dx_k^n,
\end{aligned}
\end{equation}
where we have set
\begin{equation}
\Psi^M_i(x)=e^{\frac12 \omega_i t_M} \varphi_i(x),\quad \Phi^1_i(x)=e^{-\frac12 \omega_i t_1}\varphi_i(x).
\end{equation}
Define for $n<M$, $\Psi^n_i(x)=({\cal T}_{t_M-t_n} *\Psi^M_i)(x)$ and for $n>1$, $\Phi^n_i(x)=(\Phi^1_i *{\cal T}_{t_{n}-t_1})(x)$. A simple computation gives
\begin{equation}
\Psi^n_i(x) = e^{\frac12\omega_i t_n} \varphi_i(x),\quad \Phi^n_i(x)=e^{-\frac12 \omega_i t_n} \varphi_i(x).
\end{equation}
In particular, these functions satisfy the orthogonality relation
\begin{equation}\label{eqOrtho}
\int_{\R_+} dx \Phi^n_i(x) \Psi^n_j(x)=\delta_{i,j},\quad 1\leq i,j\leq M.
\end{equation}

\subsection{Determinantal correlations}
A measure of the form \eqref{eqOnePt} forms a biorthogonal ensemble~\cite{Bor98} and thus it defines a determinantal point process. Eynard-Mehta theorem tells us that a measure of the form \eqref{eqMultiPt} is determinantal on the space $\R_+\times \{t_1,\ldots,t_m\}$, see~\cite{EM97,FN98,Jo03b,TW03b,RB04}.

Let us first consider the one-point measure, which forms a biorthogonal ensemble~\cite{Bor98}. Using the orthogonality property \eqref{eqOrtho}, one immediately obtains that the determinantal point process
\begin{equation}
\eta=\sum_{k=1}^M\delta_{x_k}
\end{equation}
has correlation kernel
\begin{equation}
K_M(x,y)=\sum_{k=1}^M \varphi_k(x)\varphi_k(y)
=\sum_{k=1}^M \frac{\Ai(-\omega_k+x) \Ai(-\omega_k+y)}{(\Ai'(-\omega_k))^2}.
\end{equation}
In particular, the distribution of the top path at time $t$, $\X_M(t)$, equals a gap probability of $\eta$ and thus it is given by a Fredholm determinant
\begin{equation}
\Pb(X_M(t)\leq s)=\det(\Id-K_M)_{L^2(s,\infty)}
\end{equation}
for any $s>0$.

Next consider the point process on $W_{M,+}\times \{t_1,\ldots,t_m\}$,
\begin{equation}
\eta=\sum_{n=1}^m\sum_{k=1}^M \delta_{(x_k^n,t_n)}.
\end{equation}
Its correlation kernel can be easily computed\footnote{See e.g.~Theorem 1.4 of~\cite{RB04} or Theorem~4.2 of~\cite{BF07} for notations closer to this paper.} and it is given by
\begin{equation}\label{eq3.22}
\begin{aligned}
K_M(x,t_i;y,t_j)&=- {\cal T}_{t_j-t_i}(x,y)\Id_{t_j>t_i}+\sum_{k=1}^M \Psi^{i}_k(x) \Phi^{j}_k(y)\\
&=\left\{\begin{array}{ll}
\sum_{k=1}^M e^{-\frac12 \omega_k (t_j-t_i)} \varphi_k(x)\varphi_k(y),&\textrm{ for }t_i\geq t_j,\\
-\sum_{k=M+1}^\infty e^{-\frac12 \omega_k (t_j-t_i)} \varphi_k(x)\varphi_k(y),&\textrm{ for }t_i<t_j.
\end{array}
\right.
\end{aligned}
\end{equation}
Then, for any $s_1,\ldots,s_m>0$,
\begin{equation}
\Pb\bigg(\bigcap_{k=1}^m \{\X_M(t_k)\leq s_k\}\bigg)=\det(\Id-\chi_s K_M)_{L^2(\R\times\{t_1,\ldots,t_m\})},
\end{equation}
where $\chi_s(t_k,x)=\Id_{x>s_k}$.

Unlike in most of the papers where the convergence to the Airy$_2$ have been proven, here we do not have a double integral representation for the correlation kernel. In this paper we need to analyze the limit of the correlation kernel using the expression in \eqref{eq3.22}.

\subsection{Scaling limit}
To understand what is the correct scaling limit for space and time that we need to consider, we first need to know how $-\omega_k$ scales for large $k$. We have (see Chapter 10.4 of~\cite{AS84})
\begin{equation}
\omega_k=f(\tfrac{3\pi}{2}(k-1/4)),\quad \Ai'(-\omega_k)=(-1)^{k-1}f_1(\tfrac{3\pi}{2}(k-1/4)),
\end{equation}
where the functions $f,f_1$ satisfy
\begin{equation}
f(z)=z^{2/3}\left[1+\frac{5}{48 z^2}+\Or(z^{-4})\right],\quad f_1(z)=\frac{z^{1/6}}{\sqrt{\pi}}\left[1+\frac{5}{48 z^2}+\Or(z^{-4})\right].
\end{equation}
Set $c_0=\frac{3^{1/3}}{2^{1/3}\pi^{2/3}}$ and $c_1=\frac{3^{2/3}\pi^{2/3}}{2^{2/3}}$. Then, for large $M$,
\begin{equation}\label{eq24}
\begin{aligned}
\omega_{[M-\lambda c_0 M^{1/3}]} &=\left[\frac{3\pi}{2}(M-\lambda c_0 M^{1/3}-1/4)\right]^{2/3}(1+\Or(M^{-2}))\\
&=c_1 M^{2/3}-\lambda+\Or(M^{-1/3},\lambda^2 M^{-2/3})
\end{aligned}
\end{equation}
as well as
\begin{equation}\label{eq25}
\Ai'(-\omega_{[M-\lambda c_0 M^{1/3}]})=\sqrt{c_0} M^{1/6}(1+\Or(\lambda M^{-2/3}).
\end{equation}
Therefore,
\begin{equation}\label{eq26}
\varphi_{[M-\lambda c_0 M^{1/3}]}(c_1 M^{2/3}+\xi)=\frac{\Ai(\xi+\lambda+\Or(M^{-1/3},\lambda^2 M^{-2/3}))}{\sqrt{c_0}M^{1/6}(1+\Or(\lambda M^{-2/3}))}\simeq \frac{\Ai(\lambda+\xi)}{\sqrt{c_0}M^{1/6}}.
\end{equation}
This implies that we need to scale space and time as
\begin{equation}
x=c_1 M^{2/3}+\xi,\quad t=2 \tau.
\end{equation}

Then the point process
\begin{equation}
\tilde \eta=\sum_{n=1}^m\sum_{k=1}^M\delta_{(x_k^n-c_1 M^{2/3},t_n)}
\end{equation}
is determinantal with (conjugated) correlation kernel given by
\begin{equation}
\tilde K_M(\xi_i,\tau_i;\xi_j,\tau_j)=e^{(\tau_j-\tau_i) c_1 M^{2/3}}K_M(c_1 M^{2/3}+\xi_i,2\tau_i;c_2 M^{2/3}+\xi_j,2\tau_j).
\end{equation}
In particular, for $s_k=c_1 M^{2/3}+S_k$ and $t_k=2\tau_k$,
\begin{equation}\label{eq4.32}
\Pb\bigg(\bigcap_{k=1}^m \{\X_M(t_k)\leq s_k\}\bigg)=\det(\Id-\chi_S \tilde K_M)_{L^2(\R\times\{\tau_1,\ldots,\tau_m\})}.
\end{equation}

For $\tau_i\geq \tau_j$ we have
\begin{equation}
\begin{aligned}
\tilde K_M(\xi_i,\tau_i;\xi_j,\tau_j)&\simeq \frac{1}{c_0 M^{1/3}}\sum_{\lambda\in I_M}e^{-\lambda(\tau_i-\tau_j)}\Ai(\lambda+\xi_i)\Ai(\lambda+\xi_j)\\
&\simeq \int_0^\infty e^{-\lambda(\tau_i-\tau_j)}\Ai(\xi_i+\lambda)\Ai(\xi_j+\lambda)
\end{aligned}
\end{equation}
where $I_M=c_0^{-1}M^{-1/3}\{0,1,\ldots,M-1\}$. Similarly, for $\tau_i<\tau_j$,
\begin{equation}
\tilde K_M(\xi_i,\tau_i;\xi_j,\tau_j)\simeq -\int_{-\infty}^0 e^{-\lambda(\tau_i-\tau_j)}\Ai(\xi_i+\lambda)\Ai(\xi_j+\lambda).
\end{equation}
In order to prove Theorem~\ref{thmMain}, we need to make the above approximations precise. In particular, we need to show the convergence of the Fredholm determinant.

\subsection{Proof of Theorem~\ref{thmMain}}
Let us recall a couple of simple bounds on the Airy function\footnote{The first bound follows by $\lim_{n\to\infty} n^{1/3} J_{[2n+u n^{1/3} u]}(2n)=\Ai(u)$ (see also (3.2.23) of~\cite{AS84}) and the bound of Landau~\cite{Lan00}. For any $x\geq 0.01$, the bound $|\Ai(x)|\leq \frac{1}{2\sqrt{\pi}x^{1/4}}e^{-\frac23 x^{3/2}}$ (see Equation~9.7.15 of~\cite{NIST:DLMF}), is better that the bound $e^{-x}$ and $e^{-0.01}>c$.}:
\begin{equation}\label{eq32}
\sup_{x\in\R}|\Ai(x)|\leq c=0.7857\ldots,\quad |\Ai(x)|\leq e^{-x}
\end{equation}
for all $x\in \R$. Also, the (absolute value of the) derivative at the zeros satisfies the lower bound $|\Ai'(-\omega_k)|\geq |\Ai'(-\omega_{k-1})|\geq |\Ai'(-\omega_1)|=0.7012\ldots$.

For $\tau_i\geq \tau_j$ the kernel is given by
\begin{multline}
\tilde K_M(\xi_i,\tau_i;\xi_j,\tau_j)
=\sum_{\lambda\in I_M} e^{-\left(\omega_{[M-\lambda c_0 M^{1/3}]}-c_1 M^{2/3}\right) (\tau_j-\tau_i)} \\ \times \varphi_{[M-\lambda c_0 M^{1/3}]}(c_1 M^{2/3}+\xi_i)\varphi_{[M-\lambda c_0 M^{1/3}]}(c_1 M^{2/3}+\xi_j)
\end{multline}
with $I_M=c_0^{-1}M^{-1/3}\{0,1,\ldots,M-1\}$, while for $\tau_i< \tau_j$ the kernel is given by
\begin{multline}\label{eq34}
\tilde K_M(\xi_i,\tau_i;\xi_j,\tau_j)
=-\sum_{\lambda\in J_M} e^{-\left(\omega_{[M-\lambda c_0 M^{1/3}]}-c_1 M^{2/3}\right) (\tau_j-\tau_i)} \\ \times \varphi_{[M-\lambda c_0 M^{1/3}]}(c_1 M^{2/3}+\xi_i)\varphi_{[M-\lambda c_0 M^{1/3}]}(c_1 M^{2/3}+\xi_j)
\end{multline}
with $J_M=c_0^{-1}M^{-1/3}\{-1,-2,\ldots\}$.

We have the following pointwise convergence of the functions entering in the expression of the kernel.
\begin{lem}\label{LemPointwisePosL} For any given $\lambda$,
\begin{equation}
\lim_{M\to\infty} \sqrt{c_0} M^{1/6} \varphi_{[M-\lambda c_0 M^{1/3}]}(c_1 M^{2/3}+\xi) = \Ai(\lambda+\xi)
\end{equation}
and
\begin{equation}
\lim_{M\to\infty}e^{-\left(\omega_{[M-\lambda c_0 M^{1/3}]}-c_1 M^{2/3}\right) (\tau_2-\tau_1)} = e^{-\lambda(\tau_1-\tau_2)}.
\end{equation}
\end{lem}
\begin{proof}
The first statement follows from \eqref{eq26}, while the second from \eqref{eq24}.
\end{proof}

In order to have convergence of the kernel, we need to be able to take the limit inside the sum. We do this by dominated convergence and therefore we need some bounds on the functions also for large values of $\lambda$.
\begin{lem}\label{lemBoundsPosL}
$ $\\
(a) For $\lambda\in [0, M^{1/6}]$, we have
\begin{equation}\label{eq37}
\left| \sqrt{c_0} M^{1/6} \varphi_{[M-\lambda c_0 M^{1/3}]}(c_1 M^{2/3}+\xi)\right|\leq C e^{-(\xi+\lambda)},
\end{equation}
for some constant $C>0$.\\
(b) For $\lambda\in (M^{1/6},M^{2/3}/c_0])$, we have
\begin{equation}
\left| \sqrt{c_0} M^{1/6} \varphi_{[M-\lambda c_0 M^{1/3}]}(c_1 M^{2/3}+\xi)\right|\leq C M^{1/6} e^{-M^{1/6}} e^{-\xi},
\end{equation}
for some constant $C>0$.\\
(c) For $\lambda\geq 0$, the exponential term satisfies
\begin{equation}
e^{-\left(\omega_{[M-\lambda c_0 M^{1/3}]}-c_1 M^{2/3}\right) (\tau_2-\tau_1)} \leq C
\end{equation}
for some constant $C>0$.
\end{lem}
\begin{proof}
(a) From \eqref{eq26} we have
\begin{equation}
\sqrt{c_0} M^{1/6} \varphi_{[M-\lambda c_0 M^{1/3}]}(c_1 M^{2/3}+\xi) = \frac{\Ai(\xi+\lambda+\Or(M^{-1/3}))}{1+\Or(M^{-1/2})}
\end{equation}
and the claimed bound follows from \eqref{eq32}.

(b) Using $\omega_{[M-\lambda c_0 M^{1/3}]}\leq \omega_{[M-c_0 M^{1/2}]}$, $|\Ai'(-\omega_{[M-\lambda c_0 M^{1/3}]})|\geq |\Ai'(-\omega_1)|\geq 0.7$, and \eqref{eq37} the claimed bound is proven.

(c) follows from \eqref{eq24}.
\end{proof}

As a consequence we get the following limit and bounds.
\begin{prop}\label{PropPosL}
For $\tau_i\geq \tau_j$, uniformly for $\xi_i,\xi_j$ in a bounded set, we have
\begin{equation}
\lim_{M\to\infty}\tilde K_M(\xi_i,\tau_i;\xi_j,\tau_j)=\int_0^\infty e^{-\lambda(\tau_i-\tau_j)}\Ai(\xi_i+\lambda)\Ai(\xi_j+\lambda)
\end{equation}
and there exists a constant $C>0$ such that
\begin{equation}\label{eqBound1}
|\tilde K_M(\xi_i,\tau_i;\xi_j,\tau_j)|\leq C e^{-(\xi_i+\xi_j)}
\end{equation}
uniformly for all $M$ large enough.
\end{prop}
\begin{proof}
It follows from Lemmas~\ref{LemPointwisePosL} and~\ref{lemBoundsPosL} by applying dominated convergence.
\end{proof}

Similarly, for $\lambda<0$ we have the following estimates.
\begin{lem}\label{lemBoundsNegL}
$ $\\
(a) For all $\lambda<0$,
\begin{equation}
\left|\sqrt{c_0} M^{1/6}\varphi_{[M-\lambda c_0 M^{1/3}]}(c_1 M^{2/3}+\xi)\right|\leq C
\end{equation}
for some constant $C>0$.\\
(b) For $-M^{1/6}<\lambda<0$,
\begin{equation}
e^{-\left(\omega_{[M-\lambda c_0 M^{1/3}]}-c_1 M^{2/3}\right) (\tau_j-\tau_i)}\leq C e^{\lambda (\tau_j-\tau_i)},
\end{equation}
for some constant $C>0$.\\
(c) For $\lambda<-M^{1/6}$,
\begin{equation}
e^{-\left(\omega_{[M-\lambda c_0 M^{1/3}]}-c_1 M^{2/3}\right) (\tau_j-\tau_i)}\leq e^{-\min\{\frac34 (-\lambda),(-\lambda)^{2/3}M^{2/9}\} (\tau_j-\tau_i)}.
\end{equation}
\end{lem}
\begin{proof}
(a) Using $|\Ai'(-\omega_{[M-\lambda c_0 M^{1/3}]})|\geq |\Ai'(-\omega_{M})|$ and $|\Ai(x)|\leq c\leq 1$ we get
\begin{equation}
\left|\sqrt{c_0} M^{1/6}\varphi_{[M-\lambda c_0 M^{1/3}]}(c_1 M^{2/3}+\xi)\right| \leq \frac{\sqrt{c_0} M^{1/6}}{|\Ai'(-\omega_{M})|}\leq C,
\end{equation}
where in the last inequality we used \eqref{eq25}.\\
(b) Denote $\tilde\lambda=-\lambda>0$. For $\tilde\lambda\leq M^{1/6}$,
$\omega_{[M+\tilde \lambda c_0 M^{1/3}]}-c_1 M^{2/3}\geq \tilde\lambda+\Or(M^{-1/3})$, from which it follows
\begin{equation}
e^{-\left(\omega_{[M+\tilde\lambda c_0 M^{1/3}]}-c_1 M^{2/3}\right) (\tau_j-\tau_i)}\leq C e^{-\tilde\lambda (\tau_j-\tau_i)}.
\end{equation}
(c) For $\tilde\lambda>M^{1/6}$, $\omega_{[M+\tilde \lambda c_0 M^{1/3}]}-c_1 M^{2/3}\geq \min\{\frac34 \tilde\lambda,\tilde\lambda^{2/3}M^{2/9}\}$, where we used the inequality $(1+x)^{2/3}\geq 1+\frac12 \min\{x^{2/3},x\}$. This gives
\begin{equation}
e^{-\left(\omega_{[M+\tilde\lambda c_0 M^{1/3}]}-c_1 M^{2/3}\right) (\tau_j-\tau_i)}\leq e^{-\min\{\frac34 \tilde\lambda,\tilde\lambda^{2/3}M^{2/9}\} (\tau_j-\tau_i)}.
\end{equation}
\end{proof}

The convergence in \eqref{eq34} is coming from the exponential term.
\begin{prop}\label{PropNegL}
For $\tau_i<\tau_j$, uniformly for $\xi_i,\xi_j$ in a bounded set, we have
\begin{equation}
\lim_{M\to\infty}\tilde K_M(\xi_i,\tau_i;\xi_j,\tau_j)=-\int_{-\infty}^0 e^{-\lambda(\tau_i-\tau_j)}\Ai(\xi_i+\lambda)\Ai(\xi_j+\lambda)
\end{equation}
and, there exists a constant $C>0$ such that
\begin{equation}\label{eqBound2}
|\tilde K_M(\xi_i,\tau_j;\xi_i,\tau_j)|\leq C
\end{equation}
uniformly for all $M$ large enough.
\end{prop}
\begin{proof}
It follows from Lemmas~\ref{LemPointwisePosL} and~\ref{lemBoundsNegL} by applying dominated convergence.
\end{proof}

We have now all the ingredient to complete the proof of our theorem. From \eqref{eq4.32} we have
\begin{equation}
\begin{aligned}
&\textrm{l.h.s.~of~}\eqref{eq51} =\lim_{M\to\infty} \det(\Id-\chi_S \tilde K_M)_{L^2(\R\times\{\tau_1,\ldots,\tau_m\})} \\
&=\lim_{M\to\infty}\sum_{n\geq 0} \frac{(-1)^n}{n!} \sum_{i_1,\ldots,i_n\in\{1,\ldots,m\}} \int_{S_{i_1}}^\infty d\xi_1 \cdots \int_{S_{i_n}}^\infty d\xi_n
\det\left[\tilde K_M(\xi_j,\tau_{i_j};\xi_k,\tau_{i_k})\right]_{1\leq j,k\leq n}\\
&=\sum_{n\geq 0} \frac{(-1)^n}{n!} \sum_{i_1,\ldots,i_n\in\{1,\ldots,m\}} \int_{S_{i_1}}^\infty d\xi_1 \cdots \int_{S_{i_n}}^\infty d\xi_n
\det\left[K_{\rm Ai}(\xi_j,\tau_{i_j};\xi_k,\tau_{i_k})\right]_{1\leq j,k\leq n}\\
&=\det(\Id-\chi_S K_{\rm Ai})_{L^2(\R\times\{\tau_1,\ldots,\tau_m\})} = \textrm{r.h.s.~of~}\eqref{eq51}.
\end{aligned}
\end{equation}
To justify the exchange of limit and sums/integrals we use dominated convergence. By Propositions~\ref{PropPosL} and~\ref{PropNegL} we have pointwise convergence of the kernel to the extended Airy kernel. To apply dominated convergence we use the bounds in Propositions~\ref{PropPosL} and~\ref{PropNegL}, together with Hadamard's bound, which says that for a $n\times n$ matrix $A$ with $|A_{i,j}|\leq 1$, $|\det(A)|\leq n^{n/2}$. This is by now very standard, see e.g.~\cite{BFP06} for detailed computations. This completes the proof of Theorem~\ref{thmMain}.

%\bibliographystyle{patplain}
%\bibliography{Biblio}

\end{document}